\documentclass{amsart}
\usepackage{mathrsfs}
\usepackage{amssymb}
\usepackage{hyperref}
\usepackage{enumitem}
\hypersetup{colorlinks=true, citecolor=RedViolet}
\usepackage[dvipsnames]{xcolor}
\newtheorem{thm}{Theorem}[section]
\newtheorem{lem}[thm]{Lemma}

\newtheorem*{ack}{Acknowledgments}
\newcommand{ \mint }{ {\displaystyle\int\hspace{-0.38cm}- }}
\newcommand{\R}{\mathbb{R}}
\newcommand{\F}{\mathbf{F}}
\newcommand{\M}{\mathcal{M}}
\newcommand{\du}{\nabla u}
\theoremstyle{definition}
\newtheorem{defn}[thm]{Definition}

\numberwithin{equation}{section}

\begin{document}
\title[Steady Stokes system]{Weighted estimates for generalized steady Stokes systems in nonsmooth domains}

\author{Sun-Sig Byun}
\address{Department of Mathematical Science,
Seoul National University, Seoul 151-747, Republic of Korea}
\email{byun@snu.ac.kr}

\author{Hyoungsuk So}
\address{Department of Mathematical Science,
Seoul National University, Seoul 151-747, Republic of Korea}
\email{scor1@snu.ac.kr}

%    General info
\subjclass[2010]{Primary 76D07; Secondary 35J47}

\date{\today.}


\keywords{Stokes system, Calder\'{o}n-Zygmund estimate, Muckenhoupt
weight, BMO coefficient, Reifenberg domain}

\begin{abstract}
We consider a generalized steady Stokes system with discontinuous
coefficients in a nonsmooth domain when the inhomogeneous term
belongs to a weighted $L^q$ space for $2<q<\infty$. We prove the
global weighted $L^q$-estimates for the gradient of the weak
solution and an associated pressure under the assumptions that the
coefficients have small BMO (bounded mean oscillation) semi-norms and the domain is
sufficiently flat in the Reifenberg sense. On the other hand, a given
weight is assumed to belong to a Muckenhoupt class. Our result generalizes the global $W^{1.q}$ estimate of Calder\'{o}n-Zygmund with respect to the Lebesgue measure for the Stokes system in a Lipschitz domain.  
\end{abstract}

\maketitle

\section{\bf Introduction}

Let $\Omega \subset \R^n$, $n\geq 2$, be a bounded domain with
nonsmooth boundary $\partial \Omega$. In this paper, we consider the
following generalized Stokes problem with inhomogeneous data:
\begin{equation}
\label{stp}
\begin{aligned}
 \left\{ \begin{array}{rlcc}
\textrm{ div } (A(x)\nabla u) -\nabla p &= \textrm{div } \F &\textrm { in }&  \Omega \\
\textrm{ div }u &= 0 &\textrm { in }&  \Omega \\
u&=0 \, &\textrm { on }&  \partial \Omega,
\end{array}
\right.
\end{aligned}
\end{equation}
where $\F=\left( \F^i_\alpha \right)^n_{i, \alpha=1}$ is a given
matrix-valued function in $L^q_\omega(\Omega)^{n^2}$ which we will
specify later, as is the tensor matrix-valued function $A= \left(
A^{\alpha \beta}_{ij} \right )^n_{i,j,\alpha, \beta=1} : \R^n \to
\R^{n^2 \times n^2}$, satisfying uniform ellipticity and boundedness,
namely; there exist positive constants $\nu$ and $L$ such that
\begin{equation}
\label{ell} \nu |\xi|^2 \leq A(x)\xi : \xi, \quad |A(x)| \leq L
\quad \forall \xi \in \R^{n^2}, \textrm{ a.e. } x \in \R^n.
\end{equation}
Here $(\cdot: \cdot)$ denotes the standard inner product in
$\R^{n^2}$ and the unknowns are the velocity $u=(u^1, \dots, u^n)$
and the pressure $p$.

The generalization of the classical steady Stokes system consists of
general second order elliptic equations in divergence form instead
of Laplace equations. This type of generalization can be found in
\cite{Fuch, Solo} and references given there. For this
generalization, we allow the tensor matrix of coefficients $A$ to be
discontinuous, but we impose a small BMO (bounded mean oscillation)
condition, as we now state.
\begin{defn}
We say that $A$ is $(\delta, R)$-vanishing if
\begin{equation}
\label{sb} \sup_{ 0<r\leq R }  \sup_{ x \in \R^n } \mint_{B_r(y)}
\left |A(x)-\overline{A}_{B_r(y)} \right |^2 dx \leq \delta^2,
\end{equation}
where $\overline{A}_{B_r(y)}=\mint_{B_r(y)} A(x) dx$ is the integral
average of $A$ over the open ball $B_r(y)$. We will clarify $\delta$
and $R$ later after Definition \ref{Reif}.
\end{defn}

Solvability and the regularity properties of solutions of the Stokes
system form the fundamental part of fluid dynamics. In particular,
there have been notable research activities on the boundary
regularity regularity in the generalized Stokes system on the
Lipschitz domain (see \cite{Gia, Macha}) and interior regularity of
Stokes system (see \cite{Brei, Dan, di2013, di2014}). In the classical
approach, which uses the representation formulas in terms of
singular operators and commutators, one needs to overcome the
obstacle coming from the non-graph domain, if one wants to deal with a
nonsmooth domain beyond the Lipschitz category. In this situation,
we cannot use directly the results obtained by using the representation
formula, so we need other approach like a maximal function, as we will
use here. The main goal of the present article is to develop a
Calder\'{o}n-Zygmund type theory for the steady Stokes system
(\ref{stp}) in the setting of weighted Sobolev and Lebesgue spaces.
This result will provide a new result in this literature, even for
the unweighted case (in standard Sobolev and Lebesgue spaces).

We now introduce the definition of a weak solution pair of the
problem (\ref{stp}).

\begin{defn}
Let $\F \in L^2(\Omega)^{n^2}$. Then $u \in W^{1, 2}_{0,
\sigma}(\Omega)^n$ is called a weak solution of the Stokes system
(\ref{stp}), if
\begin{equation}
\label{st} \int_\Omega A(x) \nabla u : \nabla \phi \, dx=
\int_\Omega \F : \nabla \phi \, dx
\end{equation}
holds for all $\phi \in W^{1, 2}_{0, \sigma}(\Omega)^{n} = \{v \in
W^{1, 2}_0(\Omega)^n : \textrm{ div }v=0\}$. If $u$ is such a weak
solution and $p \in L^2(\Omega)$ satisfies
\begin{equation*}
 \int_{\Omega} A(x)\du : \nabla \phi -  p \textrm{ div }\phi \, dx = \int_{\Omega} \F : \nabla \phi \, dx
\end{equation*}
for all $\phi \in W^{1, 2}_0(\Omega)^n$, then $(u, p)$ is called a
weak solution pair to (\ref{stp}), and $p$ is called an associated
pressure of $u$.
\end{defn}

In the next section we will return to the existence and uniqueness up to
a constant of a weak solution pair to (\ref{stp}) over a bounded
domain $\Omega$ with the following geometric regularity condition.
\begin{defn}
\label{Reif} We say that $\Omega \subset \mathbb{R}^{n}$ is
$(\delta,R)$-Reifenberg flat if for every $x\in \partial \Omega$ and
every $r\in (0,R]$, there exists an $(n-1)$ dimensional plane
$L(x,r)$ passing through $x$ such that
$$ \frac{1}{r}D[\partial \Omega \cap B_{r}(x),L(x,r)\cap B_{r}(x)]\leq \delta,$$
where $D$ denotes the Hausdorff distance which is defines as
$$D(E,F):=\max \left \{ \sup_{x\in E}dist(x,F), \sup_{y\in
F}dist(y,E) \right\},\quad E, F \in \mathbb{R}^n.$$
\end{defn}

We can assume that $R$ in both (\ref{sb}) and
Definition\ref{Reif} equals to 1 by scaling the system, while
$\delta$ is still invariant under such a scaling. Note that the
concept of $\delta$-Reifenberg flatness is a meaningful one for a
small $\delta < \frac{1}{2^{n+1}}$ (see \cite{Toro}). In this paper,
we assume $\delta$ to be a small positive constant so that
$\Omega$, $(\delta,R)$-Reifenberg flat domain, is also a
non-tangentially accessible domain (see \cite{Kenig}). In
particular, such domains are John domain (see \cite{Aika}) and then Sobolev-Poincar\'{e} inequality holds on this domain (see \cite{Buck}).  For the properties of the Reifenberg flat domain, we refer to papers \cite{Guang, Kenig, Mila, Toro}.

 We investigate this problem in weighted function spaces. More specifically, we consider Lebesgue
 spaces with respect to the measure $\omega dx$, where $\omega$ is a weight in the Muckenhoupt class $A_s$ with $1<s<\infty$.
 This is the class of nonnegative and locally integrable weight function in $\R^n$, for which
\begin{equation}
\label{aq} [\omega]_s=\sup_{y\in \R^n}\sup_{r>0}\left(
\mint_{B_r(y)}\omega(x)\, dx \right)\left(
\mint_{B_r(y)}\omega(x)^{\frac{-1}{s-1}}\, dx \right)^{s-1} <\infty.
\end{equation}
 As in \cite{Far}, typical examples of Muckenhoupt weights are
$$\omega(x)=|x|^\alpha, \quad -n < \alpha <n(q-1),$$
$$\omega(x)=\textrm{dist}(x, \mathbf{M})^\alpha, \quad -(n-j) < \alpha <n(s-j)(s-1),$$
where $\mathbf{M}$ is a compact $j$-dimensional Lipschitzian
submanifold. Therefore through choosing a particular weight function, the
advanced theory can be used for a better control of the solution, such
as in the neighborhood of a point or close to the boundary.

 Given a weight $\omega \in A_{s}$, the weighted spaces is defined as
$$L^s_\omega(\Omega)= \left \{f \in L^1_{loc}(\Omega) : \|f\|_{L^s_\omega(\Omega)}=\left( \int_\Omega |f(x)|^s\omega(x)dx \right)^{\frac 1s} < \infty \right\},$$
and we let
\begin{equation}
\label{weight} \omega(E)=\int_{E}\omega(x)\, dx.
\end{equation}

 The main result of the paper is the following Calder\'{o}n-Zygmund type estimate for a weak solution pair to (\ref{stp}).

\begin{thm}
[Main result] \label{mainthm}
 Let  $\omega \in A_{\frac q2}$ with $2<q<\infty$. Assume $\F \in L^q_\omega(\Omega)^{n^2}$. Then there exists a small constant $\delta=\delta(n, q, \nu, L, \omega)>0$ such that if $A(x)$ is $(\delta, R)$-vanishing and $\Omega \subset \R^n$,  $n \geq 2$, is $(\delta, R)$-Reifenberg flat domain, then a weak solution pair $(u, p)$ satisfies
$$ \du \in L^q_\omega(\Omega)^{n^2}, \:  p \in L^q_\omega(\Omega) \textrm{ with the estimate }$$
\begin{equation}
\label{main}\| \nabla u
\|_{L^q_\omega(\Omega)^{n^2}}+\|p\|_{L^q_\omega(\Omega)} \leq c \|
\F \|_{L^q_\omega(\Omega)^{n^2}}, \end{equation} where the constant
$c$ depends only on $n, q, \nu, L, \omega, \Omega$.
\end{thm}

 Our result is an extension of the results in \cite{byun:2008} to the context of Newtonian fluids and weighted spaces. In \cite{byun:2008}, the authors studied
\begin{equation}
\label{ellsys}
\begin{aligned}
\begin{array}{rll}
\textrm{div }(A(x)\du)&= \textrm{div }\F &\textrm{ in } \Omega \\
    u&=0 &\textrm{ on } \partial \Omega
\end{array}
\end{aligned}
\end{equation}
 under similar conditions. In that article, Calder\'{o}n-Zygmund type estimate for the weak solution to (\ref{ellsys}) was proved. This type estimate for an elliptic equation on the Reifenberg flat domain was first studied in \cite{byun:2004} and then has been extended for system and parabolic problems.

 The paper is organized as follows: In section 2, we introduce some notations and weighted Lebesgue space. Then we collect and prove the necessary statements needed for the main theorem. In section 3, we study global regularity of the gradient of weak solutions and an associated pressure to the Stokes system (\ref{stp}).

\section{\bf Preliminaries}
\subsection{\bf Notations}
\begin{enumerate}
\item $B_r(y)=\{x\in\mathbb{R}^n: |x-y|<r\}$ and $B_r=B_r(0)$.
\item $\Omega_r(y)=\Omega\cap B_r(y)$, $\Omega_r=\Omega_r(0)$, $B^+_r=B_r \cap \{x\in\mathbb{R}^n:x_n>0\}$ and $T_r=B_r \cap \{x\in\mathbb{R}^n:x_n=0\}$.
\item $\partial\Omega$ is the boundary of the domain $\Omega$, and $\partial_{w}\Omega_r(x)=\partial\Omega\cap B_r(x)$.
\item $|E|$ denotes the Lebesgue measure of the set $E\subset\mathbb{R}^n$.
\item $\overline{u}_{E}=\mint_{E}u(x)dx=\frac{1}{|E|}\int_{E}u(x)dx$ is the integral average of $u$ over $E$.
\item For vector valued function $u : \Omega \to \R^n$, we write $u \in X^n$ if each component of $u$ belongs to the function space
$X$.
\end{enumerate}

\vspace{0.2cm}

\subsection{\bf Weighted Lebesgue spaces and technical lemmas}~\\

 An important property of the Muckenhoupt weight is the relation with the Lebesgue measure as in the following lemma.
\begin{lem}
\cite{Meng} \label{wmrel} Let $E$ be measurable subset of $\Omega$
and $\omega \in A_s$ for some $1<s<\infty$.
 Then there exist positive constant $\mu$ and $\tau \in (0,1)$ independent of $B$, $E$, and $\omega$ such that
\begin{equation*}
\frac{1}{[\omega]_s} \left ( \frac{|E|}{|B|} \right)^s \leq
\frac{\omega(E)}{\omega(B)} \leq \mu \left ( \frac{|E |}{|B|}
\right)^\tau,
\end{equation*}
where $B$ is a ball and $E$ is a measurable subset of $B$.\\
\end{lem}

\begin{proof}
 From the \cite[Lemma 3.3]{Meng}, $A_s$ has strong doubling property,
$$ \omega(B)\leq [\omega]_s \left( \frac{|B|}{|E|} \right)^s \omega(E),$$
where $B$ is a ball and $E$ is a measurable subset of $B$. By
\cite[Proposition 9, V.5]{stein1993}, we can see that there are
$\omega_1$ and $\omega_2$ for each $\omega \in A_s$, $1 \leq s <
\infty$,  so that $\omega=\omega_1\omega^{1-s}_2$. Then $A_s$ has
also reverse doubling property, using \cite[Proposition 4.5,
IX]{Tor}, as the following :
$$\omega(E) \leq \mu \left( \frac{|B|}{|E|} \right)^\tau \omega(B)$$
for some $\tau \in (0,1)$. Combining these two inequalities, the
proof is completed.
\end{proof}
 We introduce the following lemma, which is derived by the standard measure theory.

\begin{lem}
\cite{Meng} \label{lem201} Suppose that $f$ is a nonnegative
measurable function in a bounded domain $\Omega$ in $\R^n$ and
$\omega \in A_s$, $1<s<\infty$. Then
\begin{equation*}
f \in L^{s}_\omega(\Omega) \quad \textrm{ if and only if } \quad S=
\sum_{k\geq1}m^{ks}\omega \left( \{x\in \Omega: f(x)>\theta
m^{k}\}\right)<\infty
\end{equation*}
for some constants $\theta>0$ and $m>1$.\\

Moreover, we have
\begin{equation*}
c^{-1}S\leq \|f\|_{L^{s}_\omega(\Omega)}^{s}\leq
c(\omega(\Omega)+S),
\end{equation*}
where $c=c(\theta, m, s)$ is a positive constant.
\end{lem}

We use the famous Hardy-Littlewood maximal function, which allows us
to control the local behavior of a function in a scaling invariant
way for qualitative study of $L^p_\omega$ functions.

\begin{defn}
\label{defn202} Let $f$ be a locally integrable function on
$\mathbb{R}^n$. Then the function
$$
(\mathcal{M}f)(x)=\sup _{r>0}\frac{1}{|B_{r}(x)|}\int
_{B_{r}(x)}|f(y)|dy$$ is called the Hardy-Littlewood maximal
function of $f$.

Further, for a function defined on a bounded $U \, \subset
\mathbb{R}^n$, we can define the Hardy-Littlewood maximal function
locally by
$$
\mathcal{M}_{U}f =\mathcal{M} \left( f  \mathbf  { \chi }
_{U}\right),$$ where $\chi$ is the standard characteristic function
on $U$.
\end{defn}

\vspace{0.1in}
 In the following lemma, we observe basic properties of the
Hardy-Littlewood maximal function.

\begin{lem}
\cite{Muck, stein1993} \label{lem203} Suppose $\omega \in A_s$ for
some $s \in (1, \infty)$, then there exists a constant $C=C(n, s,
[\omega]_s)>0$  such that
\begin{equation}
\label{202} \frac{1}{C}    \| f \|_{L^{s}_\omega(\R^n)} \leq
\|\mathcal{M}f \| _{L^{s}_\omega(\R^n)} \leq C \| f\|
_{L^{s}_\omega(\R^n)}.
\end{equation}
In particular, if $s=1$ and $\omega(x)\equiv 1$, then
$$\big| \{ x\in \R^n : (\mathcal{M})(x) > \nu \} \big | \leq \frac{C}{\nu}\int |f(x)| dx$$
for every $\nu>0$, where $C=C(n)$.
\end{lem}

For the global estimate on a $(\delta, 1)$-Reifenberg flat domain,
we use weighted version of the Vitali covering lemma as following.
We refer Lemma 3.8 in \cite{Meng} for the proof.

\begin{lem}
\cite{Meng} \label{lem206} Let $\omega \in A_s$ for some $s \in (1,
\infty)$ and let $\mathfrak{C}$ and $\mathfrak{D}$ are measurable
sets, $ \mathfrak{C} \subset \mathfrak{D} \subset \Omega$ where
$\Omega$, $(\delta, 1)$-Reifenberg flat with $0<\delta<\frac{1}{8}$,
Suppose further that there exists an $\epsilon>0$ such that
$$\omega(\mathfrak{C} \cap B_1(y) )< \epsilon  \omega(B_{1}(y)) \textnormal{ for all } y \in \Omega $$
and for all $y \in \Omega$ and for all $r \in (0,1)$ it holds
$$ B_r(y) \cap \Omega \subset \mathfrak{D} \textnormal{ if } \omega(\mathfrak{C} \cap B_r(y)) \geq \epsilon \omega(B_r(y)).$$
Then,
$$\omega(\mathfrak{C}) \leq c^* \epsilon \omega(\mathfrak{D}),$$
where $c^*=c^*(n, s, [\omega]_s)$.
\end{lem}

Another tool that makes our argument clean is scaling and
normalization. Consider the following scaled and normalized setting:
for $0<\rho<1$ and $\lambda>1$,
\begin{equation*}
\tilde{A}(x)=A(\rho x),\ \tilde{u}(x)=\frac{u(\rho
x)}{\lambda \rho},\ \tilde{p}(x)=\frac{p(\rho x)}{\lambda},\
\tilde{\F}(x)=\frac{\F(\rho x)}{\lambda},\ \textrm{and }
\tilde{\Omega}=\frac1\rho \Omega.
\end{equation*}

Then the following lemma holds.

\begin{lem}
\label{scaling}
\begin{enumerate}
\item If $(u ,p)$ is a weak solution pair to (\ref{stp}), then $(\tilde{u}, \tilde{p})$ is a weak solution pair to
\begin{equation*}
\left\{\begin{array}{rlcc}
 \textnormal{div } ( \tilde{A}(x) D\tilde{u})-\nabla \tilde{p} &= \textnormal{div } \tilde{\F}  &\textnormal{ in } & \tilde{\Omega}\\
\textnormal{div } \tilde{u}&=0 &\textnormal{ in } & \tilde{\Omega}\\
                           \tilde{u}&=0 & \textnormal{ on } & \partial\tilde{\Omega}.
       \end{array}\right.
\end{equation*}

\item If $A$ satisfies the assumptions (\ref{ell}) and (\ref{sb}), then so does $\tilde{A}$ with the same constants $\nu$ and $L$.

\item If $A$ is $(\delta, R)$-vanishing in $\Omega$, then $\tilde{A}$ is $\left(\delta, \frac{R}{\rho} \right)$-vanishing in $\tilde{\Omega}$.

\item If $\Omega$ is $(\delta, R)$-Reifenberg flat, then $\tilde{\Omega}$ is  $\left(\delta, \frac{R}{\rho} \right)$-Reifenberg flat.

\end{enumerate}
\end{lem}
\begin{proof}
 The proof follows from a direct computation.
\end{proof}
\vspace{0.2cm}

\subsection{\bf Existence and energy estimates of weak solution pairs}~\\

It is well known that if $\Omega$ is a bounded Lipschitz domain,
then the following Stokes system
\begin{equation*}
\left\{\begin{array}{rlcc}
 -\Delta u +\nabla p &= \textnormal{div } \F  &\textrm{ in } & \Omega\\
\textnormal{div } u&=0 &\textrm{ in } & \Omega\\
                           u&=0 & \textrm{ on } & \partial \Omega,
       \end{array}\right.
\end{equation*}
has a unique weak solution pair $(u,p)$ with the energy estimate
$$\| \nabla u \|_{L^2(\Omega)^{n^2}}+\|p\|_{L^2(\Omega)} \leq c \|
\F \|_{L^2(\Omega)^{n^2}},
$$
for some positive constant $c=c(\Omega, n)$, see \cite{Galdi, Sohr}.

Now we return to the generalized Stokes system (\ref{stp}). Here we
claim that the problem (\ref{stp}) also has a unique weak solution
pair with the standard estimate for our case that the underlying
domain is $(\delta, R)$-Reifenberg flat. It is well known that
$(\delta, R)$-Reifenberg flat domain is non-tangentially accessible for sufficiently small $\delta>0$,
it is a John domain as follows from \cite{Aika, Kenig}. Roughly
speaking, a domain is a John domain if it is possible to travel from
one point to another without going too close to the boundary.

We need the following two lemmas regarding a John domain.

\begin{lem}
\cite{Acos} \label{lemdiv} Let $\Omega \subset \R^n$, $n \geq 2$, be a
bounded John domain. Given $f \in L^2(\Omega)$ such that
$\int_\Omega f \, dx =0$, there exists at least one $v \in W^{1,
2}_0(\Omega)^n$ satisfying
\begin{align*}
\textnormal{div } v &= f \; \textnormal{in}\: \Omega,\\
\|\nabla v\|_{L^2(\Omega)^n} &\leq C \|f\|_{L^2(\Omega)},
\end{align*}
where $c=c(\Omega, n)$.
\end{lem}

\begin{lem}
\cite{Acos, Galdi} \label{lemp} Let $\Omega \in \R^n$, $n \geq 2$,
be a John domain. Then any bounded linear functional $\mathscr{F}$
on $W^{1, 2}_0(\Omega)^n$ identically vanishing on $W^{1, 2}_{0,
\sigma}(\Omega)^n$ is of the form $\mathscr{F}(v)=\int_\Omega p
\textnormal { div } v \, dx$ for some uniquely determined $p \in
\hat L^2(\Omega) := L^2(\Omega)/\R$.
\end{lem}

We now prove the existence and energy estimate of a weak solution
pair to (\ref{stp}).

\begin{lem}
\label{divlem} Let $\Omega \subset \R^n$, $n \geq 2$ be an open
bounded $(\delta, R)$-Reifenberg flat domain  with sufficiently small $\delta>0$ and let $\F \in
L^2(\Omega)^{n^2}$. Then there exists a unique solution pair $(u, p)
\in W^{1, 2}_{0, \sigma}(\Omega)^n \times L^2(\Omega)$ to
(\ref{stp}) satisfying $\int_\Omega p \, dx=0$ and the
standard estimate
\begin{equation}
\label{l2esti} \| \nabla u
\|_{L^2(\Omega)^{n^2}}+\|p\|_{L^2(\Omega)} \leq C \| \F
\|_{L^2(\Omega)^{n^2}},
\end{equation}
where $c=c( \Omega, n, \nu, L)$. 

In addition, if $u \in W^{1, q}_{0, \sigma}(\Omega)^n$ and $\F \in L^{q}(\Omega)^{n^2}$ for $2 \leq q<\infty$, then
\begin{equation}
\label{pesti}
\|p \|_{L^{q}(\Omega)} \leq C \left (\| \F \|_{L^{q}(\Omega)^{n^2}} + \|\nabla u \|_{L^q(\Omega)^{n^2}} \right ). 
\end{equation}
\end{lem}

\begin{proof}
It is clear that $u \in W^{1,2}_{0, \sigma}(\Omega)$ is uniquely
determined by applying Lax-Milgram theorem to (\ref{st}). Using $u
\in W^{1,2}_{0, \sigma}(\Omega)$ as the test function to (\ref{st}),
we have
\begin{equation*}
 \nu\|\nabla u\|^2_{L^2(\Omega)^{n^2}} \leq \int_\Omega A(x)\nabla u : \nabla u dx = \int_\Omega \F : \nabla u dx \leq \| \F\|_{L^2(\Omega)^{n^2}} \|\nabla u
 \|_{L^2(\Omega)^{n^2}}
\end{equation*}
and so we have
\begin{equation}
\label{uesti1} \| \nabla u \|_{L^2(\Omega)^{n^2}} \leq \frac 1\nu \|
\F\|_{L^2(\Omega)^{n^2}}.
\end{equation}

We next assume  $u \in W^{1, q}(\Omega)^n$, $2 \leq q<\infty$ consider the functional
\begin{equation}
\label{ftl}\mathscr{F}(v)  \triangleq \int_\Omega A(x)\du : \nabla v
- \F : \nabla v \ dx, \ \ \ \ \  v \in W^{1, q'}_0(\Omega)^n,
\end{equation} 
where $q'$ is H\"{o}lder conjugate.
Note that if $v \in W^{1,q}_{0, \sigma}(\Omega)$, then
$\mathscr{F}(v)=0$. We also observe from (\ref{uesti1}) that
\begin{eqnarray*}
\left| \mathscr{F}(v) \right| &\leq& c \|\du \|_{L^q(\Omega)^{n^2}}
\|\nabla v\|_{L^{q'}(\Omega)^{n^2}}+\|\F \|_{L^q(\Omega)^{n^2}}
\|\nabla v\|_{L^{q'}(\Omega)^{n^2}} \\
&\leq& c \|\F \|_{L^q(\Omega)^{n^2}} \|\nabla
v\|_{L^{q'}(\Omega)^{n^2}}
\end{eqnarray*}
for some positive $c=c(n, \nu, L)$. Consequently, $\mathscr{F}$ is a
bounded linear functional $\mathscr{F}$ on $W^{1, q'}_0(\Omega)^n$
identically vanishing on $W^{1, 2}_{0, \sigma}(\Omega)^n$. Then
since $(\delta, R)$-Reifenbefg flat domain with small $\delta$ is a John domain,  we
apply Lemma \ref{lemp} to discover that one can find a uniquely
determined $p \in L^q(\Omega)$ with $\int_\Omega p\, dx=0$ such that
\begin{equation}
\label{func1} \mathscr{F}(v)=\int_\Omega p \textrm{ div } v \ dx
\end{equation}
for all $v \in W^{1, q'}_0(\Omega)$. Then we conclude that we have a
unique solution pair $(u, p) \in W^{1,2}_{0, \sigma}(\Omega)^n
\times L^2(\Omega)$ satisfying $\int_\Omega p\, dx=0$, if $\F \in L^2(\Omega)^{n^2}$.

To prove \eqref{pesti}, we consider the problem
\begin{equation}
\label{pb1}
\begin{aligned}
&\textrm {div } v=|p|^{q-2}p-\mint_\Omega |p|^{q-2}p \,dx :=g \\
&v\in W^{1,q'}_0(\Omega)\\
&\|v\|_{W^{1, q'}(\Omega)} \leq C \|p\|^{q-1}_q.
\end{aligned}
\end{equation}

Since $\int_\Omega g \, dx=0$, $g \in L^{q'}(\Omega)$, $\|g\|_{q'} \leq c\|p \|_q^{q-1}$, from Lemma \ref{lemdiv} we deduce the existence of $v$ solving (\ref{pb1}). If we replace such a $v$ into (\ref{func1}) and use the assumption, $\int_\Omega p \, dx=0$, together with the H\"{o}lder inequality, we see
\begin{equation*}
\begin{aligned}
\|p\|^q_q &=\int_\Omega |p|^q \,dx =\int_\Omega p\left(|p|^{q-2}p-\mint_\Omega |p|^{q-2}p \, dy \right) dx \\
		&=\int_\Omega p \textrm{ div } v \, dx = \mathscr{F}(v)\\
		&=\int_\Omega A(x)\du : \nabla v + \F: \nabla v \, dx\\
		&\leq C \| \du \|_q \| \nabla v \|_{q'}+\|\F\|_q\|\nabla v \|_{q'}\\
		&\leq C \left(\| \du \|_q +\| \F \|_q \right ) \|p \|^{q-1}_{q}\\
		&\leq C \left(\| \du \|_q +\| \F \|_q \right )^q + \frac12 \|p \|^{q}_{q}\\
\end{aligned}
\end{equation*}
This inequality yields \eqref{pesti}. And \eqref{l2esti} follows from \eqref{uesti1} and \eqref{pesti}.

\end{proof}

A main point in this paper is that the nonhomogeneous term $\F$
belongs to a weighted Lebesgue space. More precisely,
$$ |\F|^2 \in L^{\frac q2}_\omega(\Omega), \ \omega \in A_{\frac q2} \textrm{ for } q \in (2, \infty),$$
which means
$$ \F \in L^{q}_\omega(\Omega)^{n^2}, \ \omega \in A_{\frac q2} \subset A_q \textrm{ for } q \in (2, \infty).$$
Using H\"{o}lder inequality and (\ref{aq}), we compute
\begin{eqnarray*}
\|\F\|^2_{L^2(\Omega)^{n^2}} &=& \int_\Omega |\F|^2\omega^{\frac 2q}\omega^{-\frac 2q}dx\\
&\leq& \left(\int_\Omega \left(|\F|^2\right)^{\frac q2}\omega \,dx\right)^{\frac2q}\left(\int_\Omega \omega^{\frac {-2}{q-2}} dx\right)^{\frac {q-2}{q}}\\
&=& \left \| |\F|^2 \right\|_{L^{\frac q2}_\omega(\Omega)^{n^2}} |\Omega|^{\frac {q-2}{q}} \left(\frac{1}{|\Omega|} \int_\Omega \omega^{\frac {-2}{q-2}} dx\right)^{\frac {q-2}{q}}\\
&\leq& \left \| |\F|^2 \right\|_{L^{\frac q2}_\omega(\Omega)^{n^2}}
|\Omega|^{-\frac {2}{q}}\omega(\Omega)[\omega]^{\frac 2q}_{\frac
q2},
\end{eqnarray*}
which implies $\F \in L^2(\Omega)^{n^2}$. This guarantees the
existence of a unique weak solution pair $(u, p)$ to (\ref{stp}).

\section{\bf Gradient estimates in $L^q_\omega$}
Throughout this section we write $c$ to mean any universal constant
that can be explicitly computed in terms of known quantities such as
$\nu, L, n, q, \omega$ and the structure of $\Omega$. Thus the exact
value may vary from line to line. If necessary, we specify it by
$c_1, \ c_2, \cdots$.

We first make interior comparison estimates. For doing this, we
consider
\begin{equation}
\label{lost} \textrm{div } \left( A \du \right)-\nabla p=\textrm{div
} \F, \ \textrm{div } u=0  \textrm { in } \Omega \supset B_6.
\end{equation}
Suppose that
\begin{equation}
\label{intdu} \mint_{B_{5}}|\du|^2+|p|^2 dx \leq 1.
\end{equation}
As usual, a weak solution to (\ref{lost}) is a function $u \in
W^{1,2}_{\sigma}(\Omega)^n$ such that
$$ \int_\Omega A\du: \nabla \phi \, dx=\int_\Omega \F: \nabla \phi\, dx$$
for all $\phi \in W^{1, 2}_{0, \sigma}(\Omega)^n$, and $(u, p)$ is a
weak solution pair if and only if $u \in W^{1,
2}_{\sigma}(\Omega)^n$ and $p \in L^2_{loc}(\Omega)$ satisfy
\begin{equation}
\label{locstp}
 \int_\Omega A\du: \nabla \phi - p\textrm{ div }\phi \, dx=\int_\Omega \F: \nabla \phi\, dx
\end{equation}
for all $\phi \in W^{1, 2}_{0}(\Omega)^n$.

We want to find local estimates of a weak solution pair to
(\ref{locstp}) in comparison with the homogeneous problem
\begin{equation}
\label{locstp2} \left\{\begin{array}{rlcc}
\textrm {div } A\nabla h -\nabla p_h &=0 &\textrm{ in } &B_{4} \\
\textrm {div } h&=0  &\textrm{ in } &B_{4} \\
h&=u  &\textrm{ on } &\partial B_{4},
\end{array}\right.
\end{equation}
and the limiting problem
\begin{equation}
\label{locstp1} \left\{\begin{array}{rlcc}
\textrm {div } \bar{A}_{B_{4}}\nabla v -\nabla p_v &=0 &\textrm{ in } &B_{3} \\
\textrm {div } v &=0  &\textrm{ in } &B_{3} \\
v &=h  &\textrm{ on } &\partial B_{3}.
\end{array}\right.
\end{equation}
 Taking the test function $h-u$ for (\ref{locstp2}) and $v-w$ for (\ref{locstp1}), respectively, and using (\ref{ell}) and (\ref{intdu}), we have
\begin{equation}
\label{inten} \mint_{B_4} |\nabla h|^2 dx \leq c \mint_{B_4} |\nabla
u|^2 dx \leq c \textrm{ and } \mint_{B_3} |\nabla v|^2 dx \leq c
\mint_{B_3} |\nabla h|^2 dx \leq c.
\end{equation}

In what follows we need the following regularity results for $v$ and
$h$.

\begin{lem}
\label{higher} Let $h \in W^{1,2}_\sigma(B_{4})^n$ be the weak
solution to (\ref{locstp2}) satisfying (\ref{intdu}). Then there
exists an exponent $r_1=r_1(\nu, L, n)>2$ such that
 $\|\nabla h \|_{L^r(B_{3})} \leq c$.
\end{lem}

\begin{proof}
By Theorem 2.2 in \cite{Gia}, there is an exponent $r_1=r_1(\nu, L,
n)>2$ such that
$$\int_{B_{3}} |\nabla h|^{r_1} dx \leq \left (\int_{B_{4}} |\nabla h|^2 dx \right)^{\frac {r_1}{2}}.$$
But by (\ref{inten}) and (\ref{intdu}), we have
$$\int_{B_{4}}|\nabla h|^2 dx \leq c \int_{B_{4}}|\nabla u|^2 dx \leq c.$$
The conclusion now follows immediately.
\end{proof}

\begin{lem}
\label{intlip} Let $(v, p_v)$ be a weak solution pair to Stokes
system (\ref{locstp1}) in $B_3$. Then there holds
$$  \|\nabla v \|_{L^\infty(B_2)^{n^2}} + \|p_v \|_{L^\infty(B_2)} \leq c.$$
\end{lem}
\begin{proof}
According to a known regularity for the limiting problem
(\ref{locstp1}), see \cite{Gia},  we have
$$ \|\nabla v \|_{L^\infty(B_2)^{n^2}} + \|p_v \|_{L^\infty(B_2)} \leq c \|\nabla v\|_{L^{2}(B_3)^n}^2.$$
Then the conclusion follows from (\ref{inten}).
\end{proof}

\begin{lem}
\label{intucp} For any $0 < \epsilon  <1$, there exists a small
$\delta=\delta(\epsilon)>0$ such that if
\begin{equation}
\label{intassum1} \mint_{B_4} |A-\bar{A}_{B_4}|^2+ | \F |^2 dx \leq
\delta^2
\end{equation}
for any weak solution pair $(u,p)$ to (\ref{st}) with (\ref{intdu}),
then one can find a weak solution pair $(v, p_v)$ to (\ref{locstp1})
in $B_3$ such that
\begin{equation*}
\mint_{B_3}|\nabla (u-v)|^2+|p-p_v|^2 dx \leq \epsilon^2.
\end{equation*}
\end{lem}

\begin{proof}
Let $(h, p_h)$ be a weak solution pair to (\ref{locstp2}). Then
$(u-h, p-p_h) \in W^{1,2}_{0, \sigma}(B_{4})\times L^2(B_{4})$ is a
weak solution pair to
\begin{equation*}
\left\{\begin{array}{rlcc}
\textrm {div } A\nabla (u-h) -\nabla (p-p_h)&=  \textrm {div } \F &\textrm{ in } &B_{4} \\
\textrm {div } (u-h)&=0  &\textrm{ in } &B_{4} \\
u-h&=0  &\textrm{ on } &\partial B_{4}.
\end{array}\right.
\end{equation*}
Using the Lemma \ref{divlem} and (\ref{intassum1}), it follows that
\begin{equation}
\label{estima1} \mint_{B_{4}} |\nabla(u-h)|^2 +  |p-p_h|^2 dx \leq
c\delta^2.
\end{equation}
For a weak solution pair $(v, v_p)$ to (\ref{locstp1}),  $(h-v, p_h
- p_v) \in W^{1,2}_{0, \sigma}(B_{3})^n \times L^2(B_3)$ is a weak
solution pair to
\begin{equation*}
\left\{\begin{array}{rlcc}
\textrm {div } \bar{A}_{B_{4}}\nabla (h-v) -\nabla (p_h - p_v)&=-\textrm{div }\big ((A-\bar{A}_{B_{4}})\nabla h \big )&\textrm{ in } &B_{3} \\
\textrm {div } (h-v)&=0  &\textrm{ in } &B_{3} \\
h-v&=0  &\textrm{ on } &\partial B_{3}.
\end{array}\right.
\end{equation*}
Then by Lemma \ref{divlem}, H\"{o}lder inequality and the
boundedness of $A(x)$, we estimate
\begin{eqnarray*}
&  &     \mint_{B_{3}}|\nabla(h-v)|^2 +|p_h-p_v|^2 dx \\
&  &    \hspace{3cm} \leq c \mint_{B_{3}}\left| A-\bar {A}_{B_{4}}\right|^2|\nabla h|^2 dx\\
&  &    \hspace{3cm} \leq c \left( \mint_{B_3}\left| A-\bar {A}_{B_{4}}\right|^{\frac {2r_1}{r_1-2}} dx  \right)^{\frac {r_1-2}{r_1}} \left(\mint_{B_{3}} |\nabla h|^{r_1} dx\right)^{\frac {2}{r_1}}\\
&  &    \hspace{3cm} \leq  c \left( \mint_{B_{3}}\left| A-\bar {A}_{B_{4}} \right|^2 \left| A-\bar {A}_{B_{4}} \right|^{\frac {4}{r_1-2}} dx  \right)^{\frac {r_1-2}{r_1}}\\
&  &    \hspace{3cm} \leq  c \left( \mint_{B_{4}}\left| A-\bar {A}_{B_{4}}\right|^2 dx  \right)^{\frac {r_1-2}{r_1}}\\
&  &    \hspace{3cm} \leq  c \delta^{2-\frac{4}{r_1}}.
\end{eqnarray*}
These estimates and (\ref{estima1}) imply
$$\mint_{B_3} | \nabla(u-v)|^2 +|p-p_v|^2 dx \leq c \left( \delta^2+\delta^{2-\frac{4}{r_1}} \right) \leq \epsilon^2,$$
by taking $\delta>0$ so small that the last inequality holds. This
finishes the proof.
\end{proof}

We next extend the interior comparison estimate obtained in Lemma
\ref{intucp} to find its boundary version. To do this, based on the
definition of the $(\delta, R)$-Reifenberg flatness, we are under
the following geometric setting
\begin{equation}
\label{domain} B_{6}^+ \subset \Omega_{6} \subset B_{6} \cap \{x_n
> -12\delta \}.
\end{equation}

From now on we consider a localized problem, the homogeneous
problem, the reference problem and a limiting problem as follows:
\begin{equation}
\label{bdstu} \left\{\begin{array}{rlcc}
\textrm{ div } (A\du) -\nabla p&= \textrm {div } \F &\textrm { in }& \Omega_6  \\
\textrm{ div }u &= 0 &\textrm { in }&  \Omega_6 \\
u&=0 \, &\textrm { on }&  \partial_w \Omega_6,
       \end{array}\right.
\end{equation}

\begin{equation}
\label{bdsth} \left\{\begin{array}{rlcc}
\textrm{ div } (A\nabla h) - \nabla p_h&= 0 &\textrm { in }& \Omega_5  \\
\textrm{ div }h &= 0 &\textrm { in }&  \Omega_5 \\
h&=u \, &\textrm { on }&  \partial \Omega_5,
       \end{array}\right.
\end{equation}

\begin{equation}
\label{bdstw} \left\{\begin{array}{rlcc}
\textrm{ div } (\bar{A}_{B_6^+}\nabla w) - \nabla p_w&= 0 &\textrm { in }&  \Omega_4  \\
\textrm{ div }w &= 0 &\textrm { in }&   \Omega_4 \\
w&=h \, &\textrm { on }&  \partial \Omega_4,
       \end{array}\right.
\end{equation}

and
\begin{equation}
\label{bdst} \left\{\begin{array}{rlcc}
\textrm{ div } (\bar{A}_{B_6^+}\nabla v) - \nabla p_v&= 0 &\textrm { in }& B_4^+  \\
\textrm{ div }v &= 0 &\textrm { in }&  B_4^+ \\
v&=0 \, &\textrm { on }&  T_4.
       \end{array}\right.
\end{equation}

$L^2$-estimates for $h$ and $w$ are derived by selecting the test
function $h-u$ for (\ref{bdsth}) and $w-h$ for (\ref{bdstw}),
respectively, and computing in a typical way along with (\ref{ell}).
We have
\begin{equation}
\label{29} \mint_{\Omega_5} |\nabla h|^2 dx \leq c \mint_{\Omega_5}
|\nabla u|^2 dx \ \textrm{ and }\ \mint_{\Omega_4} |\nabla w|^2 dx
\leq c \mint_{\Omega_4} |\nabla h|^2 dx.
\end{equation}
We further assume that
\begin{equation}
\label{bddu} \mint_{\Omega_5}|\du|^2 + |p|^2 dx \leq 1.
\end{equation}
Then by (\ref{29}), we discover that
\begin{equation}
\label{bdd} \mint_{\Omega_5}|\nabla h|^2 dx \leq c \textrm{ and }
\mint_{\Omega_4}|\nabla w|^2 dx \leq c.
\end{equation}

As in Lemma \ref{higher}, the gradient of $h$, which is the weak
solution to (\ref{bdsth}), has a higher integrability near the
boundary. This is the following lemma.

\begin{lem}
\label{bdhigher} Let $h \in W^{1,2}_\sigma(\Omega_{5})^n$ be the
weak solution to (\ref{bdsth}) with (\ref{domain}) satisfying
$\mint_{\Omega_5}|\nabla h|^2 dx \leq 1$. Then there exists an
exponent $r_2=r_2(n, \Omega, \nu, L)>2$ such that  $\|\nabla h
\|_{L^{r_2}(\Omega_{4})} \leq c$.
\end{lem}

\begin{proof}
Let $\eta$ be a standard cut off function such that $\eta \in
C^\infty_0(B_5)$, $0 \leq \eta \leq 1$, $\eta \equiv 1$ on $B_4$,
and $|\nabla \eta| \leq c$. Substituting the test function $\eta^2
h$ in (\ref{bdsth}), and using (\ref{ell}) and Young's inequality in
a standard way, we see
\begin{eqnarray}
\label{30}
\int_{\Omega_4}|\nabla h|^2 dx &\leq& c \left( \int_{\Omega_5}|h|^2|\nabla \eta|^2 dx+ \int_{\Omega_5}\eta |p-\bar{p}_{\Omega_5}| |\nabla \eta| |h|dx \right)\\
&\leq& c \left( \int_{\Omega_5}|h|^2|\nabla \eta|^2 dx+
\int_{\Omega_5} \eta^2 |p-\bar{p}_{\Omega_5}|^2 dx  \right). \nonumber
\end{eqnarray}
By Lemma \ref{lemdiv}, we can use inequality (0.8) in \cite{Gia},
which gives us
\begin{eqnarray*}
\|p-\bar{p}_{\Omega_5}\|_{L^2(\Omega_5)} &\leq& c\|\textrm{div }A(x) \nabla h \|_{W^{-1, 2}(\Omega_5)^n}\\
&\leq& c\|k \|_{W^{1,2}_0(\Omega_5)^n}\\
 &\leq& c\| \nabla k \|_{L^{2}(\Omega_5)^{n^2}},
\end{eqnarray*}
where $k \in W^{1,2}_0(\Omega_5)^n$ satisfies
$$\int_{\Omega_5} \nabla k : \nabla \phi \ dx= \int_{\Omega_5}  A(x)\nabla h : \nabla \phi dx$$
for all $\phi \in W^{1,2}_0(\Omega_5)^n$. Putting $\phi =k$, we have
$$\int_{\Omega_5} |\nabla k|^2 dx \leq c \int_{\Omega_5} |\nabla h| |\nabla k| dx.$$
Then by Young's inequality, we find
$$\int_{\Omega_5} |\nabla k|^2 dx \leq c \int_{\Omega_5} |\nabla h|^2 dx.$$
Consequently, we discover
$$\int_{\Omega_5} |p-\bar{p}_{\Omega_5}|^2 dx \leq c \int_{\Omega_5} |\nabla h|^2.$$
Since a $\delta$-Reifenberg flat domain has the measure density
condition, Sobolev inequality holds true on this domain, see
\cite{Toro}. Applying the estimate of pressure to (\ref{30}) and
using Sobolev inequality, we obtain
$$\int_{\Omega_4}|\nabla h|^2 dx \leq \frac 12 \int_{\Omega_5}|\nabla h|^2 dx + c \left( \int_{\Omega_5}|\nabla h|^{2_*} dx \right)^{\frac {2}{2_*}},$$
where $2_*= \frac{2n}{n+2}$. Then the conclusion comes from Gehring
lemma.
\end{proof}

We need the following better regularity for the limiting problem
(\ref{bdst}).

\begin{lem}
\cite{Gia} \label{bdlip} Let $(v, p_v)$ be a weak solution pair to
the Stokes system (\ref{bdst}) in $B^+_4$. Then we have
$$  \|\nabla v \|_{L^\infty(B^+_3)^{n^2}} \leq c \|\nabla v\|_{L^{2}(B_4^+)^{n^2}}$$
and
$$\| p_v \|_{L^\infty(B^+_3)} \leq c \left ( \|\nabla v\|_{L^{2}(B_4^+)^{n^2}} +   \|p_v\|_{L^{2}(B_4^+)} \right).$$

\end{lem}

\begin{lem}
\label{vcp} For any $0 < \epsilon  <1$, there exists a sufficiently
small $\delta=\delta(\epsilon)>0$ such that if $(w, p_w)$ is a weak
solution pair to (\ref{bdstw}) with (\ref{domain}) and the following
normalization condition
\begin{equation}
\label{45} \mint_{\Omega_4}|\nabla w|^2 + |p_w|^2 dx \leq 1,
\end{equation}
then there exists a weak solution pair $(v, p_v)$ to (\ref{bdst}) in
$B_4^+$ with
$$\mint_{B_4^+}|\nabla v|^2 +|p_v|^2 dx \leq 1$$
such that
$$\mint_{B^+_4}|w-v|^2 dx \leq \epsilon^2.$$
\end{lem}

\begin{proof}
We prove this lemma by contradiction. If not, then there exist
$\epsilon_0
>0$, $\{(w_k, p_{w_k})\}^\infty_{k=1}$, and
$\{\Omega^k_4\}^\infty_{k=1}$ such that $(w_k, p_{w_k}) \in
W^{1,2}_\sigma(\Omega_4^k)^n \times L^2(\Omega_4^k)$ is a weak
solution pair to
\begin{equation}
\label{46} \left\{\begin{array}{rlcc}
\textrm{div }(\bar{A}_{B^+_6} \nabla w_k) -\nabla p_{w_k}&= 0 &\textrm{ in } &\Omega_4^k \\
\textrm{div } w_k&= 0 &\textrm{ in } &\Omega_4^k \\
w_k&= 0 &\textrm{ on } &\partial_w \Omega_4^k
\end{array}\right.
\end{equation}
with
\begin{equation}
\label{47} B_{6}^+ \subset \Omega_{6}^k \subset B_{6} \cap \left
\{x_n
> -\frac{12}{k} \right \} \textrm{ and }\mint_{\Omega_4^k} |\nabla
w_k|^2 +|p_{w_k}|^2 dx \leq 1.
\end{equation}
But it holds that
\begin{equation}
\label{48} \mint_{B^+_4}|w_k-v|^2 dx > \epsilon_0^2
\end{equation}
for any weak solution $v$ to (\ref{bdst}) satisfying
$\mint_{B^+_4}|\nabla v|^2 +|p_v|^2 dx \leq 1$.

Since we can say $(w_k, p_{w_k})=0$ in $B_4 \setminus \Omega_4$ by
the zero extension from the fact $w=h=u=0$ on $\partial_w \Omega_4$,
$(w_k, p_{w_k})$ is uniformly bounded in $W^{1, 2}_\sigma(B_4)^n
\times L^2(B_4)$ in view of Poincar\'{e} inequality with (\ref{47}).
It implies that $(w_k, p_{w_k})$ is uniformly bounded in $W^{1,
2}_\sigma(B^+_4)^n  \times L^2(B_4^+)$. Thus there exists a
subsequence, which we still denote by $\{(w_k, p_{w_k})\}$, and
$(w_0, p_{w_0}) \in W^{1,2}_\sigma(B^+_4)^n \times L^2(B_4^+)$ such
that
\begin{equation}
\label{49} \left\{\begin{array}{ccc}
w_k \rightharpoonup w_0 &\textrm{ in } &W^{1,2}_\sigma(B^+_4)^n\\
w_k \to w_0 &\textrm{ in } &L^{2}(B^+_4)^n \\
p_{w_k} \rightharpoonup p_{w_0} &\textrm{ in } &L^{2}(B^+_4).
\end{array}\right.
\end{equation}

From (\ref{46}), (\ref{47}) and (\ref{48}), it follows that
\begin{equation*}
\left\{\begin{array}{rlcc}
\textrm{div }(\bar{A}_{B^+_6} \nabla w_0) -\nabla p_{w_0}&= 0 &\textrm{ in } &B_4^+ \\
\textrm{div } w_0&= 0 &\textrm{ in } &B_4^+ \\
w_0&= 0 &\textrm{ on } &T_4.
\end{array}\right.
\end{equation*}
Furthermore, it follows from (\ref{47}) and (\ref{49}) that
$$\mint_{B^+_4}|\nabla w_0|^2 + |p_0|^2  dx \leq \liminf_{k \to \infty}\mint_{B^+_4}|\nabla w_k|^2 +|p_k|^2 dx \leq 1.$$
Therefore, we reach a contradiction to (\ref{48}) by (\ref{49}).
This completes the proof.
\end{proof}

\begin{lem}
\label{bducp} For any $0 < \epsilon  <1$, there exists a
sufficiently small $\delta=\delta(\epsilon)>0$ such that if $(u, p)$
is a weak solution pair to (\ref{st}) with (\ref{domain}) and the
following normalization conditions
\begin{equation}
\label{51} \mint_{\Omega_5}|\du|^2 +|p|^2 dx \leq 1 \textnormal{ and
} \mint_{\Omega_6}|\F|^2+|A-\bar{A}_{\Omega_6}|^2dx \leq \delta^2,
\end{equation}
then there exists a weak solution pair $(v, p_v)$ of (\ref{bdst}) in
$B_3^+$ with
$$\mint_{B_4^+}|\nabla v|^2 +|p_v|^2 dx \leq 1$$
such that
$$\mint_{\Omega_2}|\du-\nabla V|^2 +|p-p_V|^2 dx \leq \epsilon^2.$$
where $V$ is the zero extension of $v$ from $B_4^+$ to $B_4$ and
$p_V$ is an associated pressure of $V$.
\end{lem}

\begin{proof}
Let $(h, p_h)$ and $(w, p_w)$ be weak solution pairs to
(\ref{bdsth}) and (\ref{bdstw}), respectively. Applying Lemma
\ref{divlem} to the system which is derived by subtracting
(\ref{bdsth}) from (\ref{bdstu}), it follows that
\begin{equation}
\label{520} \mint_{\Omega_5}|\nabla(u-h)|^2+|p-p_h|^2 dx \leq c
\mint_{\Omega_5}|\F|^2 dx \leq c\delta^2,
\end{equation}
where the last inequality comes from  (\ref{51}).

By subtracting (\ref{bdstw}) from (\ref{bdsth}), we discover
\begin{equation*}
\left\{\begin{array}{rlcc}
\textrm {div } \bar{A}_{B_{6}^+}\nabla (h-w) -\nabla (p_h - p_w)&= - \textrm{div }\big ((A-\bar{A}_{B_{6}^+})\nabla h \big )&\textrm{ in } &\Omega_{4} \\
\textrm {div } (h-w)&=0  &\textrm{ in } &\Omega_{4} \\
h-w&=0  &\textrm{ on } &\partial \Omega_{4}.
\end{array}\right.
\end{equation*}
Using Lemma \ref{divlem}, we compute
\begin{eqnarray*}
&   &    \mint_{\Omega_4}|\nabla(h-w)|^2+|p_h-p_w|^2 dx \\
&   &    \hspace{2cm} \leq c\mint_{\Omega_4}|A-\bar{A}_{B^+_6}|^2 |\nabla h|^2 dx\\
&   &    \hspace{2cm} \leq c \left( \mint_{\Omega_4}\left| A-\bar {A}_{B_{6}^+}\right|^{\frac {2r_2}{r_2-2}} dx  \right)^{\frac {r_2-2}{r_2}} \left(\mint_{\Omega_{4}} |\nabla h|^{r_2} dx\right)^{\frac {2}{r_2}}\\
&   &    \hspace{2cm} \leq c \left( \mint_{\Omega_{4}}\left| A-\bar {A}_{B_{6}^+} \right|^2 \left| A-\bar {A}_{B_{6}^+} \right|^{\frac {4}{r_2-2}} dx  \right)^{\frac {r_2-2}{r_2}}\\
&   &    \hspace{2cm} \leq  c \left( \mint_{\Omega_{6}}\left| A-\bar {A}_{B_{6}^+}\right|^2 dx  \right)^{\frac {r_2-2}{r_2}}\\
&   &    \hspace{2cm} \leq  c  \left( \mint_{\Omega_{6}}\left| A-\bar {A}_{\Omega_6}\right|^2 dx +  \left| \bar {A}_{\Omega_6} - \bar {A}_{B_{6}^+} \right|^2  \right)^{\frac {r_2-2}{r_2}}\\
&   &    \hspace{2cm} \leq c \left(\delta^2 + \mint_{B_{6}^+} \left| A-\bar {A}_{\Omega_6}\right|^2 dx \right )^{\frac {r_2-2}{r_2}} \\
&   &    \hspace{2cm} \leq  c \left(\delta^2 + \frac{|\Omega_6|}{|B_6^+|}  \mint_{\Omega_{6}} \left| A-\bar {A}_{\Omega_6}\right|^2 dx \right )^{\frac {r_2-2}{r_2}}\\
&   &    \hspace{2cm} \leq  c\big(\delta^2 + \delta^2(1+\delta)\big)^{\frac {r_2-2}{r_2}}\\
&   &    \hspace{2cm} \leq  c \left(\delta^2+\delta^3 \right)^{\frac
{r_2-2}{r_2}},
\end{eqnarray*}
where we have used H\"{o}lder inequality, Lemma \ref{bdhigher}, (\ref{ell}),
(\ref{domain}), and (\ref{51}). Then (\ref{520}) implies
\begin{equation}
\label{511} \mint_{\Omega_4} |\nabla(u-w)|^2 +|p-p_w|^2 dx \leq c
\left ( \delta^2+\delta^{2-\frac{4}{r_2}}+\delta^{3-\frac{6}{r_2}}
\right).
\end{equation}
According to (\ref{51}) and (\ref{511}), we discover that
\begin{equation}
\label{5110} \mint_{\Omega_4} |\nabla w|^2 +|p_w|^2 dx \leq c.
\end{equation}
Then we apply Lemma \ref{vcp} to find that there exists a weak
solution pair $(v, p_v)$ to (\ref{bdst}) such that
\begin{equation}
\label{52} \mint_{B_4^+} |\nabla v|^2 +|p_v|^2 dx \leq 1 \textrm{
and } \mint_{B_4^+} |w- v|^2 dx \leq \epsilon_*^2,
\end{equation}
where $\epsilon_*$ is to be determined. We extend $v$ from $B_4^+$
to $B_4$ by zero and then denote it by $V$. A direct computation and
Lemma \ref{divlem} imply that $(V, p_V)$ is a weak solution pair to
\begin{equation}
\label{53} \left\{\begin{array}{rlcc} \textrm{div }(\bar{A}_{B_6^+}
\nabla V) -\nabla p_V &=- \frac{\partial}{\partial x_n}\left(
\bar{a}^{\alpha \beta}_{nn} \frac{\partial v^\alpha}{\partial
x_n}(x',0)\chi_{\R^n_-}(x) \right)
&\textrm{ in } &\Omega_4 \\
\textrm{div } V&=0 &\textrm{ in } &\Omega_4 \\
V&= 0 &\textrm{ on } &\partial_w \Omega_4,
\end{array}\right.
\end{equation}
where $\bar{A}_{B_6^+}=\bar{a}^{\alpha \beta}_{i j}$, $v=(v^1,
\cdots, v^n)$, $x'=(x_1, \cdots, x_{n-1})$ and $\chi$ is the
standard characteristic function.

Note that $V \in W^{1, 2}_\sigma(B_4)$ and $ \nabla V = \nabla v$
a.e. in $B_4^+$, as $v=0$ on $T_4$. Then it follows from Lemma
\ref{bdlip} and (\ref{52}) that
\begin{equation}
\label{54} \|\nabla V\|_{L^\infty(\Omega_3)} = \|\nabla
v\|_{L^\infty(B^+_3)} \leq c \| \nabla v \|_{L^{2}(B_4^+)} \leq  c.
\end{equation}

It follows from (\ref{bdstw}) and (\ref{53}) that $(w-V, p_w-p_V)$
is a weak solution pair to
$$
\left\{\begin{array}{rlcc} -\textrm{div }(\bar{A}_{B_6^+} \nabla
(w-V)) +\nabla (p_w - p_V)&= \left( \bar{a}^{\alpha \beta}_{nn}
\frac{\partial v^\alpha}{\partial x_n}(x',0)\chi_{\R^n_-}(x)
\right)_{x_n} &\textrm{in} &
\Omega_4 \\
\textrm{div } (w-V)&=0&\textrm{in} &  \Omega_4 \\
w-V&=0&\textrm{on}&  \hspace{-0.3cm}\partial_w \Omega_4.
\end{array}\right.
$$
Almost similarly as in the proof of Lemma \ref{bdhigher}, we obtain
the following Caccioppoli type inequality
\begin{equation}
\label{cac} \mint_{\Omega_2}|\nabla(w-V)|^2dx \leq c
\left(\mint_{\Omega_3}|w-V|^2dx + \mint_{\Omega_3}\left
|\bar{a}^{\alpha \beta}_{nn} \frac{\partial v^\alpha}{\partial
x_n}(x',0)\chi_{\R^n_-}(x) \right |^2dx \right).
\end{equation}
The first term in right-hand side is estimated as follows.
\begin{eqnarray*}
\mint_{\Omega_3}|w-V|^2 dx  &\leq& \mint_{B_3^+}|w-v|^2dx + \frac{1}{|\Omega_3|}\int_{\Omega_3 \setminus B_3^+}|w|^2 dx\\
&\leq& \epsilon^2_* +  \frac{1}{|\Omega_3|}\left( \int_{\Omega_3 \setminus B_3^+}|w|^{\frac{2n}{n-2}}dx \right)^{\frac{n-2}{n}} \left|\Omega_3 \setminus B_3^+ \right |^{\frac 2n} \\
&\leq& \epsilon^2_* +  c\delta^{\frac 2n} \mint_{\Omega_2 }|\nabla w|^2 dx\\
&\leq& \epsilon^2_* +  c\delta^{\frac 2n},
\end{eqnarray*}
where we have used Sobolev inequality, H\"{o}lder inequality, and
(\ref{bdd}). Using (\ref{ell}), (\ref{54}) and (\ref{domain}), we
estimate
\begin{eqnarray*}
\mint_{\Omega_3}\left |\bar{a}^{\alpha \beta}_{nn} \frac{\partial v^\alpha}{\partial x_n}(x',0)\chi_{\R^n_-}(x) \right |^2dx &\leq&  \frac{c}{|\Omega_3|}\int_{\Omega_3 \setminus B_3^+}\big|\nabla v(x',0) \big|^2 dx\\
&\leq& c \frac{|\Omega_3 \setminus B_3^+|}{|\Omega_3|}\\
&\leq& c \delta.
\end{eqnarray*}
Therefore, we deduce from (\ref{cac}) that
\begin{equation}
\label{55} \mint_{\Omega_2}|\nabla(w-V)|^2 dx \leq \epsilon^2_* +
c(\delta+\delta^{\frac 2n}).
\end{equation}
Since an associated pressure is determined uniquely up to a
constant, we assume $\mint_{\Omega_2}p_w-p_V\ dx=0$. Then by Lemma
\ref{lemdiv}, we can use inequality (0.8) in \cite{Gia}, which gives
us
\begin{flalign*}
\|p_w-p_V \|&_{L^2(\Omega_2)}  \\
 \leq c &\big \| \textrm{div }(\bar{A}_{B_6^+} \nabla (w-V))- \frac{\partial}{\partial x_n}\left( \bar{a}^{\alpha \beta}_{nn} \frac{\partial v^\alpha}{\partial x_n}(x',0)\chi_{\R^n_-}(x) \right) \big \|_{W^{-1, 2}(\Omega_2)^n}\\
 \leq c &\| \nabla k \|_{L^{2}(\Omega_2)^{n^2}},
\end{flalign*}
where $k \in W^{1, 2}_0(\Omega_2)^n$ satisfies the following weak
formulation
$$\int_{\Omega_2}\nabla k : \nabla \phi \, dx = \int_{\Omega_2} \bar{A}_{B_6^+} \nabla (w-V) : \nabla \phi+ \bar{a}^{\alpha \beta}_{nn} \frac{\partial v^\alpha}{\partial x_n}(x',0)\chi_{\R^n_-}(x) \frac{\partial \phi^\beta}{\partial_{x_n}} \, dx$$
for all $ \phi \in W^{1, 2}_0(\Omega_2)^n$. Taking $\phi = k$, we
have
\begin{eqnarray*}
\int_{\Omega_2} |\nabla k|^2 dx &\leq& \int_{\Omega_2}  \bar{A}_{B_6^+} \nabla (w-V) : \nabla k+ \bar{a}^{\alpha \beta}_{nn} \frac{\partial v^\alpha}{\partial x_n}(x',0)\chi_{\R^n_-}(x) \frac{\partial k^\beta}{\partial_{x_n}} \, dx\\
&\leq& c\int_{\Omega_2} |\nabla (w-V) |^2 dx+\frac12 \int_{\Omega_2}
|\nabla k |^2 dx + c\delta |\Omega_2|.
\end{eqnarray*}
 Therefore, we have
\begin{equation}
\label{56} \mint_{\Omega_2} |p_w-p_V|^2 dx \leq c\epsilon^2_* +
c(\delta+\delta^{\frac 2n}).
\end{equation}
 Combining (\ref{55}), (\ref{56}) with (\ref{511}) and taking $\epsilon_*$ and $\delta$ small enough, we complete the proof.

\end{proof}

\begin{lem}
\label{bdlem} Given $F \in L^2(\Omega)^{n^2}$, let $(u, p) \in
W^{1,2}_{0, \sigma}(\Omega)^n \times L^2(\Omega)$ be a weak solution
pair to the steady Stokes system (\ref{stp}). Then there is a
constant $N=N(\nu, L, n)>0$ so that for any $\epsilon>0$, there
exists a small $\delta=\delta(\epsilon, \nu, L, n)>0$ such that if
$A$ is $(\delta, 42)$-vanishing, $\Omega$ is $(\delta,
42)$-Reifenberg flat, and $B_r(y)$ for $r \in (0, 1]$ and $y \in
\Omega$ satisfies
\begin{equation}
\label{61} \left |\{x \in \Omega : \mathcal{M}(|\nabla
u|^2+|p|^2)(x) \leq N^2\} \cap B_r(y) \right | \leq \epsilon
|B_r(y)|,
\end{equation}
then we have
\begin{equation}
\label{611}
 \Omega_r(y)  \subset \{x \in \Omega : \mathcal{M}(|\du|^2+|p|^2)(x)>1\} \cup \{x\in \Omega: \mathcal{M}(|\F|^2)(x) >\delta^2 \}.
\end{equation}
\end{lem}

\begin{proof}
We prove this lemma by contraposition. Assume that $B_r(y)$
satisfies (\ref{61}) and that the conclusion (\ref{611}) is false.
Then there exists a point $y_1 \in \Omega_r(y)$ such that for all
$\rho>0$,
\begin{equation}
\label{612} \mint_{\Omega_\rho(y_1)} |\nabla u|^2+|p|^2 dx \leq 1,
\quad  \mint_{\Omega_\rho(y_1)} |\F|^2 dx \leq \delta^2.
\end{equation}
We first consider the interior case that $B_{6r}(y) \subset \Omega$.
Since $B_{5r}(y) \subset \Omega_{6r}(y_1)$, we see from (\ref{612})
that
\begin{equation}
\label{621} \mint_{B_{5r}(y)}|\du|^2+|p|^2 dx\leq
\displaystyle\frac{\left|\Omega_{6r}(y_1) \right|}{\left|B_{5r}(y)
\right|}\mint_{\Omega_{6r}(y_1)}|\du|^2 + |p|^2 dx \leq \left(
\frac{6}{5}\right)^{n} < 2^n.
\end{equation}
In the same way, it follows that
\begin{equation}
\label{622} \mint_{B_{5r}(y)}|\F|^2 dx \leq 2^n \delta^2.
\end{equation}
We assume $y=0$ and then consider the rescaled functions
$$\tilde{A}(x)=\frac{A(rx)}{2^{n/2}}, \, (\tilde{u}(x), \tilde{p}(x) )= \left( \frac{u(rx)}{2^{n/2}r}, \frac{p(rx)}{2^{n/2}} \right), \, \tilde{\F}(x)=\frac{\F(rx)}{2^{n/2}}, \textrm{ and } \tilde{\Omega}=\frac 1r \Omega.$$
With this setting, it is not difficult to see that all the
assumptions of Lemma \ref{intucp} are satisfied by Lemma
\ref{scaling}, (\ref{621}) and (\ref{622}). Then according to Lemma
\ref{intucp} and Lemma \ref{intlip}, after scaling back, we find
that there exists a pair $(v, p_v) \in W^{1, 2}_\sigma(B_{3r}^+)
\times L^2(B_{3r}^+)$ such that
\begin{equation}
\label{623} \|\nabla v \|_{L^\infty(B_{2r})^{n^2}}+ \| p_v
\|_{L^\infty(B_{2r})}\leq N_0 \textrm{ and } \mint_{B_{3r}}|\nabla
(u-v)|^2+|p-p_v|^2 dx \leq c_* \epsilon^2
\end{equation}
for some positive constant $N_0=N_0(n, \nu, L)$, where $c_*$ is to
be determined in a universal way. We write $N_1=max\{ 2N_0, 2^{n/2}\}$
to discover that
\begin{equation*}
\{ x \in B_r : \mathcal{M}(|\du|^2+|p|^2) > N_1^2 \} \subset \{ x
\in B_r : \mathcal{M}_{B_{3r}} (|\nabla (u-v)|^2+|p-p_v|^2) > N_0^2
\}.
\end{equation*}
From this inclusion, Lemma \ref{lem203} and (\ref{623}), we conclude
\begin{align*}
\frac{1}{|B_r|} \big | \{ x\in B_r : &\mathcal{M} \big( |\du|^2 +|p|^2 \big)(x) > N_1^2  \} \big | \\
\leq  &\frac{1}{|B_r|} \left|\left\{x\in B_{r} : \mathcal{M}_{B_{4r}}\left(|\nabla (u- v)|^2 +|p-p_v|^2 \right)(x)>N_1^{2} \right\}\right|\\
\leq  &  c \mint_{B_{r}}|\nabla(u- v)|^{2} +|p-p_v|^2  dx\\
\leq  & cc_*\epsilon^2 < \epsilon,
\end{align*}
by taking sufficiently small $c_*$ in order to have the last
inequality. This is a contradiction to (\ref{61}).

We next consider the boundary case that $B_{6r}(y) \not\subset
\Omega$. In this case, there is a boundary point $y_0 \in \partial
\Omega \cap B_{6r}(y)$. By the Reifenberg flatness condition and the
small BMO condition, we assume that there exists a new coordinate
system through suitable orientation and translation, depending on
$y_0$ and $r$, so that in this new coordinate system, the origin is
$y_0+\delta_0\overrightarrow{n_0}$ for some small $\delta_0>0$ and
for some inward unit normal $\overrightarrow{n_0}$. We now denote
the variable as $z$ in the new coordinate, $y_0=z_0$ and $y_1=z_1$.
The we find
\begin{equation}
\label{631} B_{42r}^+ \subset \Omega_{42r} \subset \{z \in B_{42r} :
z_n>-84r\delta\}
\end{equation}
and
\begin{equation}
\label{632} \mint_{\Omega_{42r}} \left
|A(z)-\overline{A}_{\Omega_{42r}} \right |^2 dz \leq \delta^2.
\end{equation}

Moreover, it follows from (\ref{612}) that
\begin{equation}
\label{633} \mint_{\Omega_{35r}}|\du|^2+|p|^2 dx \leq
\frac{|B_{42r}|}{|B_{35r}^+|} \mint_{\Omega_{42r}}|\du|^2+|p|^2 dx
\leq 2 \left( \frac 65 \right)^n  < 2^{n+1}
\end{equation}
and
\begin{equation}
\label{633-1} \mint_{\Omega_{35r}}|\F|^2 dx \leq 2^{n+1} \delta^2.
\end{equation}

We apply Lemma \ref{scaling} to $\rho=7r$ and
$\lambda=2^{\frac{n+1}{2}}$ to see that all the assumptions of Lemma
\ref{bducp} are satisfied by (\ref{631}), (\ref{632}) and
(\ref{633}). As a consequence, we find that there exists a function
$V \in W^{1,2}_\sigma(\Omega_{28r})$ such that
$$\|\nabla V\|_{L^\infty(\Omega_{21r})^{n^2}}+\| p_V \|_{L^\infty(\Omega_{21r})} \leq
N_2$$ for some constant $N_2=N_2(\nu, L, n)$ and
$$ \mint_{\Omega_{7r}}|\nabla(u-V)|^2 + |p-p_V|^2 dx \leq c_1 \epsilon,$$
where $c_1$ is to be determined.

As in the interior case, putting $N_3= \max \{2N_2, 2^{\frac n2}
\}$, we conclude
$$\frac{1}{|B_{7r}|} \left| \{z\in \Omega : \mathcal(|\du|^2+|p|^2)>N_3^2\}\cap B_{7r} \right | \leq cc_1 \epsilon,$$
which implies that
$$\frac{1}{|B_{r}|} \left| \{z\in \Omega : \mathcal(|\du|^2+|p|^2)>N_3^2\}\cap B_{r} \right | \leq cc_1 \epsilon<\epsilon,$$
by taking $c_1$ so that the last inequality holds. Finally, we set
$N=\max\{N_1, N_3\}$ to complete the proof.
\end{proof}

\begin{lem}
\label{bdlemw} Assume that $\omega \in A_s$ for some $s \in (1,
\infty)$. Given $F \in L^2_\omega(\Omega)^n$, let $(u, p) \in
W^{1,2}_{0, \sigma}(\Omega)^n \times L^2(\Omega)$ be a weak solution
pair to the steady Stokes system (\ref{stp}). Then there is a
constant $N=N(\nu, L, n)>0$ so that for any $\epsilon>0$ there
exists a small $\delta=\delta(\epsilon, \nu, L, q, \omega)>0$ such
that if $A$ is $(\delta, 42)$-vanishing, $\Omega$ is $(\delta,
42)$-Reifenberg flat, and $B_r(y)$ for $r \in (0, 1]$ and $y \in
\Omega$ satisfies
\begin{equation}
\label{70} \omega \left (\{x \in \Omega : \mathcal{M}(|\nabla
u|^2+|p|^2)(x) \leq N^2\} \cap B_r(y) \right ) \leq \epsilon \omega
\left ( B_r(y) \right),
\end{equation}
then we have
\begin{equation}
\label{71}
 \Omega_r(y)  \subset \{x \in \Omega : \mathcal{M}(|\du|^2+|p|^2)(x)>1\} \cup \{x\in \Omega: \mathcal{M}(|\F|^2)(x) >\delta^2 \}.
\end{equation}
\end{lem}

\begin{proof}
From Lemma \ref{wmrel} and (\ref{70}), we have
\begin{align*}
\big| \{x \in \Omega : \mathcal{M}(|\du|^2 + &|p|^2) >N^2\} \cap B_r(y) \big |\\
\geq &\left(\frac 1\mu \frac{\omega(\{ x\in \Omega : \mathcal{M}(|\du|^2 + |p|^2) >N^2\} \cap B_r(y) ) }{\omega(B_r(y))}\right)^{\frac 1\tau}|B_r(y)|\\
\geq &\left(\frac{\epsilon}{\mu}\right)^{\frac 1\tau}|B_r(y)|.
\end{align*}
We use Lemma \ref{bdlem} with $\epsilon$ replaced by
$\left(\frac{\epsilon}{\mu}\right)^{\frac 1\tau}$, to find
$\delta=\delta(\epsilon, \nu, L, n, \omega, s)$ so that (\ref{71})
holds.
\end{proof}

We are now all set to prove the main result. \vspace{0.2cm}

\textbf{\textit{Proof of Theorem \ref{mainthm}}}. We first assert
that
\begin{equation}
\label{80} \|\du\|_{L^q_\omega(\Omega)^{n^2}} + \| p
\|_{L^q_\omega(\Omega)} \leq c, \; \textrm{ if } \| \F
\|_{L^q_\omega(\Omega)^{n^2}} \leq \delta
\end{equation}
for some constant $c=c(n, q, \nu, L, \Omega, \omega)$.
To do this,
we compute
\begin{equation}
\label{801} \|\F\|^2_{L^2(\Omega)^{n^2}} =
\int_\Omega|\F|^2\omega^{\frac 2q}\omega^{-\frac 2q}dx \leq
\left(\int_\Omega |\F|^q \omega \, dx \right)^{\frac 2q} \left(
\int_\Omega \omega^{\frac{-2}{q-2}}dx \right)^{\frac{q-2}{q}}.
\end{equation}
Since $\Omega$ is bounded, there is a ball $B_{\frac d2} (x_0)
\supset \Omega$ for some $x_0 \in \Omega$, where $d$ is the diameter
of $\Omega$. Using (\ref{aq}) and (\ref{weight}), we estimate
\begin{flalign*}
\left( \int_\Omega \omega^{\frac{-2}{q-2}}dx \right)^{\frac {q-2}{2}} &\leq \left( \int_{B_{\frac d2}(x_0)} \omega^{\frac{-2}{q-2}}dx \right)^{\frac {q-2}{2}}\\
& \hspace{-2cm}=  \left( \mint_{B_{\frac d2}(x_0)} \omega \, dx \right) \left( \mint_{B_{\frac d2}(x_0)} \omega \, dx \right)^{-1} \left( \mint_{B_{\frac d2}(x_0)} \omega^{\frac{-2}{q-2}}dx \right)^{\frac q2-1}\left| B_{\frac d2}\right|^{\frac q2-1}\\
&\hspace{-2cm} \leq \frac{\left| B_{\frac d2} \right|^{\frac
q2}}{\omega \left(B_{\frac d2}(x_0) \right)}[\omega]_{\frac q2}\leq
\frac{d^{\frac{nq}{2}}|B_1|^{\frac q2}}{\omega(\Omega)}[\omega]_{\frac
q2}.
\end{flalign*}
Thus from (\ref{80}) and (\ref{801}), we have
\begin{equation}
\label{81-1} \|\F\|^2_{L^2(\Omega)^{n^2}} \leq
\frac{d^n|B_1|}{\omega(\Omega)^{\frac 2q}}[\omega]^{\frac 2q}_{\frac
q2} \delta^2.
\end{equation}

We now take $\epsilon \in (0, 1)$ and $N$ and choose the
corresponding $\delta$ given by Lemma \ref{bdlemw}. Then write
\begin{equation*}
\begin{array}{c}
\mathfrak{C} =\left\{x\in\Omega: \M \left(|\du|^2 +|p|^2
\right)(x)>N^2 \right\} \textrm{ and }\\ \mathfrak{D}
=\left\{x\in\Omega: \M \left(|\du|^2 + |p|^2 \right)(x)>1 \right\}
\cup \left\{x\in\Omega: \M \left(|\F|^2\right)(x)>\delta^2 \right\}.
\end{array}
\end{equation*}
By using Lemma \ref{lem203}, (\ref{l2esti}) and (\ref{81-1}), one
can check that the first hypothesis of Lemma \ref{lem206} as
follows.
\begin{eqnarray*}
|\mathfrak{C}\cap B_1(y)| & \leq &  c \int_{\Omega}|\du|^2 +|p|^2 dx\\
& \leq & c \int_{\Omega}|\F|^2 dx \\
& \leq & c\delta^{2}\\
& \leq & \left( \frac\epsilon\mu  \right)^{\frac 1\tau} |B_1|,
\end{eqnarray*}
by choosing a small enough $\delta$, if necessary, in order to get
the last inequality. Then Lemma \ref{wmrel} implies
$$\omega \big( \mathfrak{C} \cap B_1(y) \big) \leq \mu \left(\frac{|\mathfrak{C} \cap B_1(y)|}{|B_1|} \right)^\tau \omega \big(B_1(y) \big) \leq \epsilon \omega \big(B_1(y) \big).$$

On the other hand, the second hypothesis of Lemma \ref{lem206}
follows directly from Lemma \ref{bdlemw}. Therefore thanks to Lemma
\ref{lem206}, we have
\begin{equation}
\begin{aligned}
\label{81}
\omega  &\left( \left\{x\in\Omega:\M \left(|\du|^2 +|p|^2 \right)(x)>N^2 \right\} \right) \\
	 &\leq \epsilon_1 \omega \left(\left \{x \in \Omega: \M
\left(|\du|^2 +|p|^2 \right)(x)>1 \right \} \right)+ \epsilon_1 \omega
\left( \left \{x \in \Omega: \M \left(|\F|^2\right)(x)>\delta^2 \right\}\right)
\end{aligned}
\end{equation}
for $\epsilon_1=c^*\epsilon$, where $c^*$ depends only on $n, q,
[\omega]_{\frac q2}$. Using an iteration argument from (\ref{81}),
we further have the following power decay estimate.
\begin{flalign}
\label{82}
\omega \bigg(\{x\in\Omega:\M (|\du|^2&+|p|^2 )(x)>N^{2k} \} \bigg) \nonumber\\
&\leq \epsilon_1^k \omega \big(\left\{x\in\Omega:\M \left(|\du|^2 +|p|^2 \right)(x)>1 \right \} \big)\\
&+ \displaystyle\sum_{i=1}^{k}\epsilon_1^i \omega \left(\left
\{x\in\Omega:\M \left(|\F |^2 \right)(x)>\delta^2 N^{(k-i)2}\right\}
\right). \nonumber
\end{flalign}
Then using this estimate (\ref{82}), we estimate
\begin{flalign*}
\sum_{k=1}^{\infty}N^{qk} &   \omega \bigg( \big\{x\in \Omega:\M (|\du|^2+|p|^2)(x)>N^{2k}\big\} \bigg) \\
\leq  &\sum_{k=1}^{\infty} (N^{q}\epsilon_1)^k \omega \left(\left\{x\in\Omega:\M \left(|\du|^2+|p|^2 \right)(x)>1 \right\} \right) \\
&   \ \  + \sum_{i=1}^{\infty} (N^{q}\epsilon_1)^i \underbrace{ \sum_{k=i}^{\infty} N^{q(k-i)} \omega \left(\left \{x\in\Omega:\M \left(|\F |^2 \right)(x)>\delta^2 N^{(k-i)2}\right\}\right)}_{S}\\
 \leq & \sum_{k=1}^{\infty} (N^{q}\epsilon_1)^k \omega
(\Omega)+\sum_{i=1}^{\infty} (N^{q}\epsilon_1)^i \  S.
\end{flalign*}

We next show that $S$ is finite.  In light of Lemma \ref{lem201},
Lemma \ref{lem203} and the assumption that $\| \F
\|_{L^q_\omega(\Omega)^{n^2}} \leq \delta$, we compute
\begin{equation*}
S \leq c \frac 1\delta \| \F \|_{L^{q}_\omega(\Omega)^{n^2}} \leq c
\end{equation*}
for some $c=c(n, q, \nu, L, \omega, \Omega)$. Consequently, we
discover
$$ \sum_{k=1}^{\infty}N^{qk} \omega \left(\left\{x\in\Omega:\M \left(|\du|^2+|p|^2 \right)(x)>N^{2k}\right\} \right)  \leq c\sum_{k=1}^{\infty} (N^q\epsilon_1)^{k} \leq c,$$
by selecting $\epsilon$ so small that $N^q\epsilon_1 < 1$.
Therefore, the  assertion (\ref{80}) is now proved by Lemma
\ref{lem201}.

To derive the desired estimate (\ref{main}) in Theorem
\ref{mainthm}, we consider the normalized functions as
$$u_\lambda=\frac{u}{\lambda}, \ p_\lambda=\frac{p}{\lambda} \textrm{ and } \F_\lambda=\frac{\F}{\lambda},$$
where $\lambda=\delta^{-1}\| \F \|_{L^q_\omega(\Omega)^{n^2}}$. Then
it follows that
$$\| \F_\lambda \|_{L^q_\omega(\Omega)^{n^2}} \leq \delta.$$
Then (\ref{80}) implies that there is a constant $c=c(n, q, \nu, L,
\omega, \Omega)$ such that
$$\| \du_\lambda \|_{L^q_\omega(\Omega)^{n^2}}+\| p_\lambda \|_{L^q_\omega(\Omega)} \leq c,$$
which is (\ref{main}). This completes the proof.

\begin{ack}
 S. Byun was supported by the National Research Foundation of Korea (NRF) grant funded by the Korea Government (NRF-2015R1A4A1041675). H. So was supported by the National Research Foundation of Korea (NRF) grant funded by the Korea Government (NRF-2015R1A2A1A15053024).
\end{ack}

\bibliographystyle{amsplain}

\end{document}